\documentclass[11pt]{amsart}
\usepackage[latin1]{inputenc}
\usepackage{amssymb}
\usepackage{pdfsync}
\usepackage[english]{babel}

\usepackage[pdftex, colorlinks=true,urlcolor=blue,linkcolor=blue,citecolor=blue]{hyperref}
\usepackage{graphicx}

\usepackage[capitalize]{cleveref}
\usepackage{fullpage}
\usepackage{color}
\usepackage{amsmath}
\usepackage{amsfonts}
\usepackage{mathrsfs}
\usepackage{t1enc , graphicx}
\usepackage{verbatim}
\usepackage{bbm}
\usepackage{enumitem}

\usepackage[left= 1.3 in, right= 1.3 in ,top= 1 in, bottom = 2.1 in]{geometry}

\newcommand{\R}{\mathbb{R}}

\newcommand{\C}{\mathbb{C}}
\newcommand{\T}{\mathbb{T}}
\newcommand{\Q}{\mathbb{Q}}
\newcommand{\Z}{\mathbb{Z}}
\newcommand{\N}{\mathbb{N}}

\newtheorem{theorem}{Theorem}[section]
\newtheorem*{theorem*}{Theorem}
\newtheorem{proposition}[theorem]{Proposition}
\newtheorem{lemma}[theorem]{Lemma}

\newtheorem{corollary}[theorem]{Corollary}
\newtheorem*{corollary*}{Corollary}
\newtheorem{question}[theorem]{Question}

\theoremstyle{definition}
\newtheorem*{definition*}{Definition}
\newtheorem{definition}[theorem]{Definition}

\theoremstyle{remark}

\newtheorem*{remark*}{Remark}
\newtheorem{remark}{Remark}

\begin{document}
\author{Anh N. Le}
\address{Department of Mathematics\\
	Ohio State University\\
	231 W. 18th Ave., Columbus, OH 43210}
\email{le.286@osu.edu}

\title{Sublacunary sets and interpolation sets for nilsequences}
\maketitle

\begin{abstract}
    A set $E \subset \mathbb{N}$ is an interpolation set for nilsequences if every bounded function on $E$ can be extended to a nilsequence on $\mathbb{N}$. Following a theorem of Strzelecki, every lacunary set is an interpolation set for nilsequences. We show that sublacunary sets are not interpolation sets for nilsequences. Furthermore, we prove that the union of an interpolation set for nilsequences and a finite set is an interpolation set for nilsequences. Lastly, we provide a new class of interpolation sets for Bohr almost periodic sequences, and as the result, obtain a new example of interpolation set for $2$-step nilsequences which is not an interpolation set for Bohr almost periodic sequences. 
\end{abstract}

\section{Introduction}
A \emph{Bohr almost periodic sequence} is a uniform limit of trigonometric polynomials, i.e. the functions of the form $\psi(n) = \sum_{j=1}^k c_j e^{2 \pi i n \alpha_j}$ where $\alpha_j \in \R$ for $n \in \N$. A subset $E$ of $\N$ is an \emph{$I_0$-set} (or \emph{interpolation set for Bohr almost periodic sequences}) if every bounded function on $E$ can be extended to a Bohr almost periodic sequence. Equivalently, every bounded function on $E$ is the restriction of the Fourier transform of a discrete measure on the torus $\T = \R/\Z$. The class of $I_0$-sets have been extensively studied in harmonic analysis since 1960s \cite{Hartman_1961,Hartman_Ryll-Nardzewski_1964, Ryll-Nardzewski_1964, Graham_Hare_2013}. One notable result is due to Strzelecki \cite{Strzelecki_1963} where he showed that every lacunary set \footnote{$\{r_n\}_{n \in \N} \subset \N$ with $r_1 < r_2 <\ldots$ is \emph{lacunary} if $\inf_{n \in \N} \frac{r_{n+1}}{r_n} > 1$.} is an $I_0$-set. Recently, the author \cite{Le_2019_1} proved that sublacunary sets are not $I_0$-sets. Here $\{r_n\}_{n \in \N} \subset \N$ with $r_1 < r_2 <\ldots$ is \emph{sublacunary} if $\lim_{n \to \infty} (\log r_n)/n = 0$.

Under the harmonic analytic point of view, Bohr almost periodic sequences are the Fourier transforms of discrete measures on $\T$. On the other hand, from the dynamical systems perspective, Bohr almost periodic sequences are the evaluations of continuous functions on the orbits of toral rotations. With this standpoint, Bergelson, Host and Kra \cite{Bergelson_Host_Kra05} introduced a generalization of Bohr almost periodic sequences and named them nilsequences (See \cref{sec:nilsequenes_background} for definition). Under this generalization, Bohr almost periodic sequences coincide with $1$-step nilsequences. Since their introduction, nilsequences have become important objects in ergodic theory, arithmetic combinatorics and number theory. They are indispensable ingredients in the recent development in ergodic Ramsey theory \cite{Leibman15, Frantzikinakis_Host_2017}. They also played prominent roles in the programs of finding linear patterns in primes \cite{Green_Tao10, Green_Tao12a, Green_Tao_Ziegler12}, and more recently in the progress of Chowla and Sarnak's conjectures \cite{Sarnak12, Tao_Teravainen_17, Frantzikinakis_Host_2018}.


Roughly speaking, interpolation sets for a class of sequences tell us which coordinates of the sequences that are pairwise independent. Hence, understanding their interpolation sets can shed light on the inherent structures of the family of sequences we are studying. Because of this reason and since nilsequences are important objects in various areas of mathematics, it is of interest to study interpolation sets for nilsequences. Here is the precise definition.
\begin{definition}
    A set $E \subset \N$ is called an \emph{interpolation set for $k$-step nilsequences} (or \emph{$k$-step interpolation set}) if for every bounded function $f: E \to \C$, there exists a $k$-step nilsequence $\psi: \N \to \C$ such that $\psi(n) = f(n)$ for $n \in E$. 
    
    A set $E$ is called an \emph{interpolation set for nilsequences} if $E$ is a $k$-step interpolation set for some $k \in \N$. 
\end{definition}
Since Bohr almost periodic sequences are $1$-step nilsequences, every $I_0$ set is an interpolation set for nilsequences. Therefore, following from Strzelecki's result \cite{Strzelecki_1963}, lacunary set are interpolation sets for nilsequences. In \cite{Le_2019_1}, the author asked whether there is a sublacunary set that is an interpolation set for nilsequences. The main result of this paper give a negative answer to that question.
\begin{theorem}
\label{thm:main}
    Sublacunary sets are not interpolation sets for nilsequences. 
\end{theorem}
It was shown in \cite{Ryll-Nardzewski_1964} that the class of $I_0$-sets is closed under the unions with finite sets. Here we prove an analogous result regarding interpolation sets for nilsequences.
\begin{theorem}
\label{thm:union_with_finite_set}
    If $E \subset \N$ is a $k$-step interpolation set and $S \subset \N$ is finite, then $E \cup S$ is a $k$-step interpolation set.
\end{theorem}
Known examples of $I_0$-sets include lacunary sets and some suitable unions of lacunary sets, like $\{2^n\}_{n \in \N} \cup \{2^n + 1\}_{n \in \N}$. In the next proposition, we provide a new example of $I_0$-set.
\begin{proposition}
\label{prop:union_of_lacunary_set_with_a_shift}
    Let $\{r_n\}_{n \in \N} \subset \N$ be lacunary and let $\{t_n\}_{n \in \N} \subset \N$ be such that $\lim_{n \to \infty} t_n/r_n = 0$. Then $E = \{r_n\}_{n \in \N} \cup \{r_n + t_n\}_{n \in \N}$ is an $I_0$-set if and only if $\{t_n\}_{n \in \N}$ is not a set of Bohr recurrence.\footnote{A set $R \subset \N$ is a \emph{set of Bohr recurrence} if and only if for any finite dimensional torus $\T^d$ and any $\alpha \in \T^d$, we have $0$ is in the closure of $R \alpha = \{n \alpha: n \in R\}$. For example, the set of even natural numbers is a set of Bohr recurrence, but the set of odd natural numbers is not.} 
\end{proposition}
From \cref{prop:union_of_lacunary_set_with_a_shift}, we have following examples of $I_0$-sets: $\{2^n\}_{n \in \N} \cup \{2^n + 3n + 1\}_{n \in \N}, \{3^n\}_{n \in \N} \cup \{3^n +n^2 + 2\}_{n \in \N}$, and $\{n!\}_{n \in \N} \cup \{n! + 2n - 1\}_{n \in \N}$. The class of Bohr almost periodic sequences is strictly contained in the class of $2$-step nilsequences. In \cite{Le_2019_1} we provided an example of $2$-step interpolation set that is not an $I_0$-set. As a corollary of \cref{prop:union_of_lacunary_set_with_a_shift}, we obtain another example for set of this type.
\begin{corollary}
\label{cor:2-step-1-step}
    Let $\{r_n\}_{n \in \N} \subset \N$ be lacunary. Let $\{t_n\}_{n \in \N} \subset \N$ be such that $\lim_{n \to \infty} t_n/r_n = 0$ and $\{t_n\}_{n \in \N}$ is a set of Bohr recurrence. Then $\{r_n\}_{n \in \N} \cup \{r_n + t_n\}_{n \in \N}$ is a $2$-step interpolation set but not an $I_0$-set. 
\end{corollary}
Using \cref{prop:union_of_lacunary_set_with_a_shift}, we also derive a different proof of Corollary 2.7.9 in \cite{Graham_Hare_2013} which says that the sum of a lacunary set and a finite set is an $I_0$-set (\cref{cor:lac_union_shift}). 

\section{Background}

\subsection{Nilsequences}
\label{sec:nilsequenes_background}
    Let $G$ be a $k$-step nilpotent Lie group and $\Gamma \subset G$ be a discrete, cocompact subgroup. The homogeneous space $X = G/\Gamma$ is called a \emph{$k$-step nilmanifold}. For any $g \in G$, the map $g: X \to X$ defined by $x \mapsto gx$ is a diffeomorphism and the pair $(X, g)$ is called a \emph{$k$-step nilsystem}.
    
    For $x \in X$ and a continuous function $F$ on $X$, the sequence $(F(g^n x))_{n \in \N}$ is called a \emph{basic $k$-step nilsequence}. Examples of basic $k$-step nilsequences include $(e^{2 \pi i n^k \alpha})_{n \in \N}$ where $\alpha \in \R$, or more generally, $(e^{2 \pi i P(n)})_{n \in \N}$ where $P \in \R[x]$ of degree not greater than $k$. 
    
    A sequence $\psi: \N \to \C$ is a \emph{$k$-step nilsequence} if for every $\epsilon > 0$, there exists a basic $k$-step nilsequence $\psi_{\epsilon}: \N \to \C$ such that $|\psi(n) - \psi_{\epsilon}(n)| < \epsilon$ for all $n \in \N$. In other words, a $k$-step nilsequence is a uniform limit of basic $k$-step nilsequences.
    
    The space of $k$-step nilsequences is closed under pointwise addition, multiplication, complex conjugation and uniform limits. 
    
    More details on nilsequences can be found in \cite[Section 4.3.1]{Bergelson_Host_Kra05} and \cite[Section 11.3.2]{host_kra_18}.
    
    \subsection{Special representations of nilsequences}
    We can always embed a $k$-step nilmanifold $X = G/\Gamma$ into a $k$-step nilmanifold $X' = G'/\Gamma'$ where $G'$ is connected and simply connected and every element of $G$ is represented in $G'$ (for example, see \cite[Corollary 26, Section 10.5.1]{host_kra_18}). Therefore, every $k$-step nilsequence can be realized from a $k$-step nilmanifold with a connected and simply connected Lie group.

\subsection{Characterization of $I_0$-sets}
    For $A \subset \N$ and $\alpha \in \T^d = \R^d/\Z^d$, let $\overline{A \alpha}$ denote the closure of $A \alpha := \{a \alpha: a \in A\}$ in $\T^d$.

    Two sets $A, B \subset \N$ are said to be \emph{separable by a Bohr rotation} if there exists a finite dimensional torus $\T^d$ and an element $\alpha \in \T^d$ such that $\overline{A \alpha} \cap \overline{B \alpha} = \varnothing$.
\begin{theorem}[Hartman-Ryll-Nardzewski \cite{Hartman_Ryll-Nardzewski_1964}]
    \label{thm:criterion_bohr_interpolation}
    A set $E \subset \N$ is an $I_0$-set if and only if every two disjoint subsets of $E$ are separable by a Bohr rotation.
\end{theorem}

Strzelecki \cite{Strzelecki_1963} proved that all lacunary sets are $I_0$-sets.
By Hartman-Ryll-Nardzewski characterization \cite{Hartman_Ryll-Nardzewski_1964}, some unions of lacunary sets are $I_0$-sets, for example, $\{2^n\}_{n \in \N} \cup \{2^n + 1\}_{n \in \N}$. In fact, it is shown in \cite{Graham_Hare_2013} that any finite union of shifts of lacunary sets is an $I_0$-set, for instance, $\bigcup_{i=1}^{\ell} \{2^n + i\}_{n \in \N}$.  Grow \cite{Grow_1987} constructed a class of $I_0$-sets which are not finite unions of lacunary sets, for example, $\{3^{n^2} + 3^j: n \geq 1, (n-1)^2 \leq j \leq n^2\}$ (see also M\'ela \cite{Mela_1969}). More information on $I_0$-sets can be found in the book \cite{Graham_Hare_2013} by Graham and Hare.

\subsection{Characterization of interpolation sets for nilsequences}
For $A \subset \N$, for a $k$-step nilsystem $(X = G/\Gamma, g)$ and $x \in X$, let $g^A x$ denote the set $\{g^a x: a \in A\}$ and let $\overline{g^A x}$ denote the closure of $g^A x$ in $X$. 
 \begin{definition*}
        Two sets $A, B \subset \N$ are said to be \emph{separable by a $k$-step nilrotation} if there exists a $k$-step nilsystem $(X = G/\Gamma, g)$ and $x \in X$ such that $\overline{g^A x} \cap \overline{g^B x} = \varnothing$.
    \end{definition*}
Analogous to Hartman-Ryll-Nardzewski's criterion \cite{Hartman_Ryll-Nardzewski_1964}, we have a characterization of $k$-step interpolation sets and its proof is identical to  \cite[Theorem 3.1]{Le_2019_1}.
\begin{theorem}
\label{thm:criterion_nilinterpolation}
    The set $E \subset \N$ is a $k$-step interpolation set if and only if every two disjoint subsets of $E$ are separable by a $k$-step nilrotation.
\end{theorem}

\subsection{Sets of pointwise nilrecurrence}

\begin{definition}
\label{def:nilrecurrence}
A set $R \subset \N$ is called a \emph{set of pointwise $k$-step nilrecurrence} if for any $k$-step nilsystem $(X = G/\Gamma, g)$ and for any $x \in X$, we have $x \in \overline{g^R x}$.
\end{definition}

Let $\pi: G \to X$ be the natural quotient map $\pi(g) = g \Gamma$. Let $1_G$ be the identity element of $G$ and $1_X = \pi(1_G)$. In \cref{def:nilrecurrence}, by changing the base point to $1_X$, we have an equivalent definition: $R \subset \N$ is a set of pointwise $k$-step nilrecurrence if and only if for any $k$-step nilsystem $(X, g)$, we have $1_X \in \overline{g^R 1_X}$.

For any $k \in \N$, the set $k \N$ is a set of pointwise $k$-step nilrecurrence. Other than that, little is known about sets of pointwise nilrecurrence. However, there is a related and better-studied notion called sets of topological $k$-step nilrecurrence. A set $R \subset \N$ is called a \emph{set of topological $k$-step nilrecurrence} if for any $k$-step nilsystem $(X, g)$ and for any $U \subset X$ open, there exists $r \in R$ such that $U \cap g^{-r} U \neq \varnothing$.\footnote{Usually when defining sets of topological recurrence, we restrict to minimal systems. However, as every orbit closure in nilsystems is minimal, we do not need this restriction here.} It is proved in \cite{Host_Kra_Maass_2016} that a set $R$ is a set of topological $k$-step nilrecurrence if and only if it is a set of Bohr recurrence.  

It is obvious that every set of pointwise $k$-step nilrecurrence is a set of topological $k$-step nilrecurrence. But the converse is probably false (See \cite{Pavlov_2008}.) For example, while $\{n^2\}_{n \in \N}$ is a set of topological nilrecurrence, it is not clear whether it is a set of pointwise $2$-step nilrecurrence or not.


\subsection{Nilmanifolds and Mal'cev bases}
\label{sec:malcev_basis}
Let $X = G/\Gamma$ be a $k$-step nilmanifold with connected and simply connected $G$. If $m$ is the dimension of $G$ and $\mathcal{G}$ is the Lie algebra of $G$, then $\mathcal{G}$ admits a base $(\xi_1, \ldots, \xi_m)$ satisfying the following properties:
\begin{enumerate}
    \item The map $\psi: \R^m \to G$ defined by
    \[
        \psi(t_1, \ldots, t_m) = \exp(t_1 \xi_1) \cdots \exp(t_m \xi_m)
    \]
    is a diffeomorphism from $\R^m$ onto $G$.
    
    \item $\psi(\Z^m) = \Gamma$.
\end{enumerate}
The base $(\xi_1, \ldots, \xi_m)$ is called a \emph{Mal'cev basis} of $G$. If $g = \psi(t_1, \ldots, t_m)$, then $t_1, \ldots, t_m$ are called the \emph{Mal'cev coordinates} of $g$ in the base $(\xi_1, \ldots, \xi_m)$. 

For $\underline{s} = (s_1, \ldots, s_m)$ and $\underline{t} = (t_1, \ldots, t_m) \in \R^m$, write
\[
    \underline{s} * \underline{t} = \psi^{-1} (\psi(\underline{s}) \psi(\underline{t}))
\]
There exist $m-1$ polynomials $P_1, \ldots, P_{m-1}$ with $P_i \in \Q[s_1, \ldots, s_i, t_1, \ldots, t_i]$ of degree at most $k-1$ such that for all $\underline{s}, \underline{t} \in \R^m$,
\[
    \underline{s} * \underline{t} = (s_1 + t_1, s_2 + t_2 + P_1(s_1, t_1), \ldots, s_m + t_m + P_{m-1}(s_1, \ldots, s_{m-1}, t_1, \ldots, t_{m-1})))
\]
We have $\psi([0,1)^m)$ is a fundamental domain for $X$. Hence, we can identify $X$ with $[0,1)^m$ through the diffeomorphism $\widetilde{\psi}: X \to [0,1)^m$. 

We then have a chain of maps:
\[
    \R^m \xrightarrow{\psi} G \xrightarrow{\pi} X \xrightarrow{\tilde{\psi}} [0,1)^m,
\]
and define $\varphi: \R^m \to [0,1)^m$ to be $\varphi = \tilde{\psi} \circ \pi \circ \psi$. Equivalently, for $\underline{x} \in \R^m$, there exists a unique $\underline{z}(x) \in \Z^m$ such that $\underline{x} * \underline{z}(x) \in [0, 1)^m$. Then we can see that $\varphi(\underline{x}) = \underline{x} * \underline{z}(x)$. Further details on Mal'cev bases can be found in \cite{Malcev_1949, Corwin_Greenleaf90} and \cite[Chapter 10]{host_kra_18}.

\begin{definition}
For a real number $x$, let $\lVert x \rVert$ denote the distance from $x$ to the nearest integer. For $\underline{x} = (x_1, \ldots, x_m)$ and $\underline{y} = (y_1, \ldots, y_m) \in \R^m$, define 
\[
    \lVert \underline{x} - \underline{y} \rVert_{T^m} = \max_{1 \leq i \leq m} \lVert x_i - y_i \rVert.
\]
\end{definition}
We can see that $\lVert \cdot \rVert_{\T^m}$ defines a metric on $[0,1)^m$ which is compatible with the Euclidean metric. The reason we choose this metric only to make our computations later less cumbersome. We use $d$ to denote the metric on $X$ that comes from the metric $\lVert \cdot \rVert_{\T^m}$ on $[0,1)^m$ through the map $\widetilde{\psi}^{-1}$.

\subsection{Notation}
    We use $\N$ to denote the set of positive integers $\{1, 2, 3, \ldots\}$ and use $\T$ to denote the torus $\R/\Z$. The notation $(b(n))_{n \in \N}$ represents a sequence of complex numbers and $\{r_n\}_{n \in \N}$ represents a subset of $\N$ with $r_1 < r_2 < \ldots$.
   
    Let $f, g: \R \to \R$. By writing $f(x) = O_{a, b}(g(x))$, we mean there exist positive constants $C$ and $N_0$ that depend on $a$ and $b$ such that $|f(x)| < C g(x)$ for every $x > N_0$.  

\section{Union of interpolation sets with finite sets}

\begin{lemma}
\label{lem:removing_finite_set}
    Let $R \subset \N$ be a set of pointwise $k$-step nilrecurrence and $A \subset R$ is a finite set. Then $R \setminus A$ is still a set of pointwise $k$-step nilrecurrence.
\end{lemma}
\begin{proof}
    By way of contradiction, assume $R \setminus A$ is not a set of pointwise $k$-step nilrecurrence. Then there exists a $k$-step nilsystem $(X = G/\Gamma, g)$ such that 
    \[
        1_X \not \in \overline {g^{R \setminus A} 1_X}.
    \]
    Let $\alpha \in \T$ be an irrational number. Since $A$ is finite, we have $0 \not \in \overline{A \alpha}$.
    Consider the $k$-step nilsystem $(X \times \T, (g, \alpha))$. It then follows that the closure
    \[
        \overline{(g, \alpha)^R 
        (1_X, 0)} = \overline{(g, \alpha)^A (1_X, 0)} \cup \overline{(g, \alpha)^{R \setminus A} (1_X, 0)}
    \]
    does not contain $(1_X, 0)$. Therefore $R$ is not a set of pointwise $k$-step nilrecurrence, and this is a contradiction. 
\end{proof}

\begin{lemma}
\label{lem:finite_recurrence}
    Let $R$ be a set of pointwise $k$-step nilrecurrence. Let $X = G/\Gamma$ be a $k$-step nilmanifold with the corresponding metric $d$. Suppose $H$ is a compact subset of $G$. For every $\epsilon > 0$, there exists a finite set $R_{\epsilon} \subset R$ such that the following holds: For all $h \in H$, there exists $r \in R_{\epsilon}$ for which $d(h^r 1_X, 1_X) < \epsilon$.
\end{lemma}
\begin{proof}
    Suppose $R = \{r_n\}_{n \in \N} \subset \N$.
    By contradiction, assume there exists $\epsilon > 0$ such that for all $N \in \N$, there exists $h_N \in H$ satisfying
    \[
        d(h_N^{r_n} 1_X, 1_X) \geq \epsilon \text{ for all } n \leq N.
    \]
    Since $H$ is compact, the sequence $(h_N)_{N \in \N}$ has an accumulation point in $G$, say $g$. It follows that
    \[
        d(g^{r_n} 1_X, 1_X) \geq \epsilon \text{ for all } n \in \N.
    \]
    But this contradicts the hypothesis that $R$ is a set of pointwise $k$-step nilrecurrence.
\end{proof}

\begin{lemma}
\label{lem:nilrecurrence_partition}
A set of pointwise $k$-step nilrecurrence can be partitioned into two sets of pointwise $k$-step nilrecurrence.
\end{lemma}

\begin{proof}
It is well-known that up to isomorphism, the set of pair $(G, \Gamma)$ where $G$ is a $k$-step nilpotent Lie group and $\Gamma$ is a discrete, cocompact subgroup is countable (for example see \cite[Section 10.5.2]{host_kra_18}). Let $\mathcal{C}$ be a set of representatives of those non-isomorphic pairs $(G, \Gamma)$. For each $(G, \Gamma) \in \mathcal{C}$, we can partition $G$ into countably many compact subsets, say $G = \bigcup_{j=1}^{\infty} H_{G,j}$. It follows that the set $\mathcal{D} = \{(G, \Gamma, H_{G,j}, \epsilon): (G, \Gamma) \in \mathcal{C}, j \in \N, \epsilon \in \{1/n:n \in \N\} \}$ is countable. Enumerate $\mathcal{D}$ as $\{(G_n, \Gamma_n, H_n, \epsilon_n): n \in \N\}$.

Let $R$ be a set of pointwise $k$-step nilrecurrence. Invoking \cref{lem:finite_recurrence}, there exists a finite set $A_1 \subset R$ such that for every $h \in H_1$, there exists $r \in A_1$ for which
\begin{equation}
\label{eq:A_1}
    d(h^r 1_{G_1/\Gamma_1}, 1_{G_1/\Gamma_1}) < \epsilon_1.
\end{equation}
Since $A_1$ is finite, by \cref{lem:removing_finite_set}, $R \setminus A_1$ is still a set of pointwise $k$-step nilrecurrence. Using \cref{lem:finite_recurrence} again, we can find a finite set $B_1 \subset R \setminus A_1$ for which for every $h \in H_1$, there exists $r \in B_1$ satisfying \eqref{eq:A_1}. 

Removing $A_1$ and $B_1$ from $R$ and applying above procedure repeatedly, we get two sequences of pairwise disjoint subsets of $R$, say $(A_n)_{n \in \N}$ and $(B_n)_{n \in \N}$, such that the following holds: For every $h \in H_n$, there exists $r_1 \in A_n$ and $r_2 \in B_n$ for which
\[
    d(h^{r_1} 1_{G_n/\Gamma_n}, 1_{G_n/\Gamma_n}) < \epsilon_n
\]
and
\[
    d(h^{r_2} 1_{G_n/\Gamma_n}, 1_{G_n/\Gamma_n}) < \epsilon_n.
\]
Let $A = \bigcup_{n \in \N} A_n$ and $B = \bigcup_{n \in \N} B_n$. By construction, $A$ and $B$ are disjoint subset of $R$. It remains to show that $A$ and $B$ are sets of pointwise $k$-step nilrecurrence.

By definition of $\mathcal{D}$, for any $k$-step nilmanifold $X = G/\Gamma$, every $g \in G$ and $\epsilon > 0$, there is an $n \in \N$ such that $(G, \Gamma)$ is isomorphic to $(G_n, \Gamma_n)$, and under this isomorphism $g \in H_n$ and $\epsilon_n < \epsilon$. Therefore, there exists $r \in A_n \subset A$ such that $d(g^r 1_{G/\Gamma}, 1_{G/\Gamma}) < \epsilon_n < \epsilon$. Because $\epsilon$ is arbitrary, we deduce that $A$ is a set of pointwise recurrence for the system $(G/\Gamma, g)$. Since $G, \Gamma, g \in G$ are arbitrary, we have $A$ is a set of pointwise $k$-step nilrecurrence. Similar conclusion is true for $B$.
\end{proof}

The following two lemmas are standard and their proofs are included for completeness.
\begin{lemma}
\label{lem:dec-31-3}
    Let $A, B, C \subset \N$. Suppose that $A$ is separable by some $k$-step nilrotations from $B$ and $C$. Then $A$ is separable by a $k$-step nilrotation from $B \cup C$. 
\end{lemma}
\begin{proof}
    Let $(X_1 = G_1/\Gamma_1, g_1)$ be a $k$-step nilsystem such that $\overline{g_1^A 1_{X_1}} \cap \overline{g_1^B 1_{X_1}} = \varnothing$. Let $(X_2 = G_2/\Gamma_2, g_2)$ be a $k$-step nilsystem such that $\overline{g_2^A 1_{X_2}} \cap \overline{g_2^C 1_{X_2}} = \varnothing$. Considering the product system $(X_1 \times X_2, (g_1, g_2))$, we have
    \[
        \overline{(g_1, g_2)^A (1_{X_1}, 1_{X_2})} \cap \overline{(g_1, g_2)^{B \cup C} (1_{X_1}, 1_{X_2})} = \varnothing.
    \]
    Therefore, $A$ and $B \cup C$ are separable by a $k$-step nilrotation. 
\end{proof}


\begin{lemma}
\label{lem:dec-31-4}
    If $E$ and $F$ are $k$-step interpolation sets and they are separable by a $k$-step nilrotation, then $E \cup F$ is a $k$-step interpolation set.
\end{lemma}
\begin{proof}
     Let $A, B$ be two disjoint subsets of $E \cup F$. Then $A = (A \cap E) \cup (A \cap F)$ and $B = (B \cap E) \cup (B \cap F)$. Invoking \cref{thm:criterion_bohr_interpolation}, the sets $A \cap E$ and $B \cap E$ are separable by a $k$-step nilrotation because they are disjoint subsets of a $k$-step interpolation set $E$. On the other hand, $A \cap F$ and $B \cap E$ are separable by a $k$-step nilrotation because $E$ and $F$ are separable by a $k$-step nilrotation. By \cref{lem:dec-31-3}, the sets $A$ and $B \cap E$ are separable by a $k$-step nilrotation. Similarly, $A$ and $B \cap F$ are separable by a $k$-step nilrotation. Invoking \cref{lem:dec-31-3} again, we have $A$ and $B$ are separable by a $k$-step nilrotation. Since $A$ and $B$ are two arbitrary disjoint subsets of $E \cup F$, we get $E \cup F$ is a $k$-step interpolation set by \cref{thm:criterion_bohr_interpolation}.
\end{proof}

We are ready to prove \cref{thm:union_with_finite_set}.
\begin{proof}[Proof of \cref{thm:union_with_finite_set}]
    By induction, it suffices to assume $S = \{a\}$ for some $a \in \N$. By contradiction, suppose $E \cup \{a\}$ is not $k$-interpolation. Since both $E$ and $\{a\}$ are $k$-interpolation, by \cref{lem:dec-31-4}, it follows that $E$ and $\{a\}$ are not separable by any $k$-step nilrotation. Therefore, for every $k$-step nilsystem $(X = G/\Gamma, g)$, $g^a 1_X$ is in the closure of $g^E 1_X$. Or equivalently, $1_X$ is in the closure of $g^{E-a} 1_X$. Since $(X = G/\Gamma, g)$ is an  arbitrary $k$-step nilsystem, we conclude that $E - a$ is a set of pointwise $k$-step nilrecurrence. In view of \cref{lem:nilrecurrence_partition}, the set $E-a$ can be partitioned into two sets of pointwise $k$-step nilrecurrence, say $A$ and $B$. Hence, for every $k$-step nilsystem $(X = G/\Gamma, g)$, we have
    \[
        1_X \in \overline{g^A 1_X} \cap \overline{g^B 1_X}
    \]
    Or equivalently,
    \[
        g^a 1_X \in \overline{g^{A + a} 1_X} \cap \overline{g^{B+a} 1_X} 
    \]
    
    In particular, $\overline{g^{A + a} 1_X} \cap \overline{g^{B+a} 1_X}$ is non-empty. Since $(X = G/\Gamma, g)$ is an arbitrary $k$-step nilsystem, we deduce that $A + a$ and $B+a$ are not separable by any $k$-step nilrotation. Hence, by \cref{thm:criterion_bohr_interpolation}, the set $E = (A + a) \cup (B + a)$ is not $k$-step interpolation. This is a contradiction. 
\end{proof}

\section{Sublacunary sets and interpolation sets}

\subsection{Partitioning a compact set by a system of polynomial equations}
\label{sec:partition_compact_set}

\begin{lemma}
\label{lem:poly_partition_space}
    Let $m, b, \ell \in \N$ and $Q_1, \ldots, Q_{\ell} \in \R[x_1, \ldots, x_m]$ of degree at most $b$. Then the union of the solutions sets of $\ell$ equations $Q_i(x_1, \ldots, x_m) = 0$ for $1 \leq i \leq \ell$ divides $\R^m$ into less than $(2 b \ell)^m$ regions.
    
\end{lemma}
\begin{remark}
    As an illustration of the lemma above, let $Q_1, Q_2 \in \R[x_1, x_2]$ of degree at most $2$. Then for each $i \in \{1, 2\}$ the solution set of $Q_i(x_1, x_2) = 0$ is a conic section. It is easy to see that the number of regions that two conic sections divide $\R^2$ into is not exceeding $13$, which is less than $(2 \cdot 2 \cdot 2)^2$ (The maximum number $13$ is achieved when we have two hyperbolas.) 
\end{remark}
\begin{proof}
    On the boundary of every region, there is a point which is a solution to a system of the form
    \[
        \begin{cases}
            Q_{i_1}(x_1, \ldots, x_m) = 0 \\
            \ldots \\
            Q_{i_m}(x_1, \ldots, x_m) = 0
        \end{cases}
    \]
    where $i_1, \ldots, i_m \in \{1, \ldots, \ell\}$ distinct. By B\'ezout's theorem, each system has maximum $d^m$ solutions counting with multiplicity and counting also $\infty$. There are ${k \choose m} < k^m$ such systems. Hence, there are less than $(kd)^m$ points in $\R^m$ that are solutions of one of those systems. 
    
    Each point is on the boundaries of at most $2^m$ regions where each region is corresponding to a choice of $Q_{i_j} > 0$ or $Q_{i_j} < 0$ for $1 \leq j \leq m$. It follows that there are at most $(2 b \ell)^m$ regions. 
\end{proof}

For the rest of \cref{sec:partition_compact_set}, we fix a $k$-step nilmanifold $X = G/\Gamma$ with $G$ being connected and simply connected.  Assume $\dim G = m$ and the multiplication rule on the Mal'cev coordinates is the following:
    \begin{equation}
    \label{eq:multiplicative_rule_Malcev_basis}
        \underline{s} * \underline{t} = (s_1 + t_1, s_2 + t_2 + P_1(s_1, t_1), \ldots, s_m + t_m + P_{m-1}(s_1, \ldots, s_{m-1}, t_1, \ldots, t_{m-1})))
    \end{equation}
    where for $1 \leq i \leq m-1$, the polynomial $P_i \in \Q[s_1, \ldots, s_i, t_1, \ldots, t_i]$ has degree at most $k-1$. We also fix the maps $\psi, \pi, \widetilde{\psi}$ and  $\varphi$ accordingly as they are defined in \cref{sec:malcev_basis}.

All constants and big-O terms found in this section implicitly depend on $G, \Gamma$, the choice of Mal'cev's basis, maps $\psi, \pi, \widetilde{\psi}, \varphi$, and polynomials $P_1, \ldots, P_{m-1}$.  

\begin{lemma}
\label{lem:nov-20-1}
    There exists $R \in \R[x]$ having all coefficients non-negative and for $n \in\N, 1 \leq i \leq m$, there exist $Q_{i,n} \in \Q[x_1, \ldots, x_m]$ satisfying the followings:
    \begin{enumerate}[label=(\alph*)]
        \item For every $\underline{x} = (x_1, \ldots, x_m) \in \R^m$, 
        \begin{equation}
        \label{eq:formula_for_Malcev}
            \underline{x}^{* n} = \underbrace{\underline{x} * \cdots * \underline{x}}_{n} = (Q_{1,n}(\underline{x}), \ldots, Q_{m,n}(\underline{x})).
        \end{equation}
        \item For $1 \leq i \leq m$ and $n \in \N$, the polynomial $Q_{i,n}$ has degree less than $k^m$.
        \item For $1 \leq i \leq m$ and $n \in \N$, the sum of absolute values of coefficients of $Q_{i,n}$ is bounded above by $R(n)$. 
    \end{enumerate}
\end{lemma}
\begin{proof}
    The existence of $Q_{i,n}$ satisfying \eqref{eq:formula_for_Malcev} follows from the multiplication rule on Mal'cev coordinates. By induction, we can compute $Q_{i,n}$ explicitly.
    \[
        Q_{1,n}(\underline{x}) = n x_1,
    \]
    \[
        Q_{2,n}(\underline{x}) = n x_2 + \sum_{j=1}^{n-1} P_1(x_1, j x_1).
    \]
    In general, for $1 \leq i \leq m$ and $n \in \N$,
    \[
        Q_{i,n}(\underline{x}) = n x_i + \sum_{j=1}^{n-1} P_{i-1}(x_1, \ldots, x_{i-1}, Q_{1, j}(\underline{x}), \ldots, Q_{i-1, j}(\underline{x})) \in \Q[x_1, \ldots, x_m].
    \]
   It also follows that
    \[
        \deg Q_{i,n} \leq \max \{1, \deg P_{i-1} \times \max_{1\leq \ell \leq i-1, 1 \leq j \leq n-1}\{\deg Q_{\ell,j}\}\}.
    \]
    By induction, we get $\deg Q_{i,n} \leq \prod_{j=1}^{i-1} (\deg P_j + 1)$. Since $\deg P_j \leq k -1$, we have 
    \[
        \deg Q_{i,n} \leq k^{i-1} < k^m.
    \]
    
    It remains to prove $(c)$. Let $P \in \R[s_1, \ldots, s_{m-1}, t_1, \ldots, t_{m-1}]$ with the coefficients for any monomial being the maximum of the absolute values of coefficients for corresponding monomials in any of $P_1, \ldots, P_{m-1}$. For $i \in \{1, \ldots, m\}$, $n \in \N$, define polynomials $\tilde{Q}_{i,n} \in \Q[x_1, \ldots, x_m]$ recursively as follows: $\tilde{Q}_{1,n}(\underline{x}) = n x_1$, and for $i \geq 1$,
    \[
        \tilde{Q}_{i,n}(\underline{x}) = n x_i + \sum_{j=1}^{n-1} P(x_1, \ldots, x_{i-1}, \underbrace{1, \ldots, 1}_{m-i}, \tilde{Q}_{1, j}(\underline{x}), \tilde{Q}_{2, j}(\underline{x})  \ldots, \tilde{Q}_{i-1, j}(\underline{x}), \underbrace{1, \ldots, 1}_{m-i}).
    \]
    Then all the coefficients of $\tilde{Q}_{i,n}$ are  non-negative and greater than or equal to the absolute value of the corresponding coefficients in $Q_{i,n}$. 
    
    Now the sum of coefficients of $\tilde{Q}_{i,n}$ is equal to
    \begin{equation}
    \label{eq:nov-16-1}
         \tilde{Q}_{i,n}(\underline{1}) = n + \sum_{j=1}^{n-1} P(\underbrace{1, \ldots, 1}_{m-1}, \tilde{Q}_{1, j}(\underline{1}), \tilde{Q}_{2, j}(\underline{1}), \ldots, \tilde{Q}_{i-1, j}(\underline{1}), \underbrace{1, \ldots, 1}_{m-i}). 
    \end{equation}
    From above formula, it is not hard to see that $\tilde{Q}_{i,n}(\underline{1}) \leq \tilde{Q}_{i,n + 1}(\underline{1})$ and $\tilde{Q}_{i,n}(\underline{1}) \leq \tilde{Q}_{i + 1,n}(\underline{1})$ for all $i \leq m$ and $n \in \N$. Hence the right hand side of \eqref{eq:nov-16-1} is not greater than
    \[
        n + (n-1) P(\underbrace{1, \ldots, 1}_{m-1}, \underbrace{\tilde{Q}_{i-1, n}(\underline{1}), \ldots, \tilde{Q}_{i-1, n}(\underline{1})}_{m-1}).
    \]
    Define $S \in \Q[x]$ by $S(x) = x + (x-1) P(\underbrace{1, \ldots, 1}_{m-1}, \underbrace{x, \ldots, x}_{m-1})$. We then have
    \[
        \tilde{Q}_{i,n}(\underline{1}) \leq S(\tilde{Q}_{i-1,n}(\underline{1}))
    \]
    for all $1 \leq i \leq m$ and $n \in \N$. By induction, 
    \[
        \tilde{Q}_{i,n}(\underline{1}) \leq S (S (\ldots (S(\tilde{Q}_{1,n}(\underline{1}) \ldots ) = S (S (\ldots (S(n) \ldots )
    \]
    where the polynomial $S$ is repeated $m-1$ times. Let $R(n) = S(S( \ldots(S(n) \ldots))$ where polynomial $S$ is repeated $m-1$ times. Then we have $\tilde{Q}_{i,n}(\underline{1}) \leq R(n)$. This finishes our proof. 
\end{proof}

In the rest of \cref{sec:partition_compact_set}, we fix polynomials $Q_{i,n}$ (for $i \in \{1, \ldots, m\}$ and $n \in \N$) found in \cref{lem:nov-20-1}. 

\begin{proposition}
\label{prop:compute_z_i}
    For $\underline{x} := (x_1, \ldots, x_m) \in \R^m$, let $\underline{z} = \underline{z}(x) := (z_1(x), \ldots, z_m(x))$ be the unique element in $\Z^m$ be such that $\underline{x} * \underline{z}(x) \in [0,1)^m$. There exists a constant $c_1 > 0$ such that
    \[
        |z_i(x)| = O((\max_{1 \leq i \leq m}|x_i|)^{c_1})
    \]
    as $\underline{x} \in \R^m$.
\end{proposition}
\begin{proof}
    Using \eqref{eq:multiplicative_rule_Malcev_basis}, we can compute $\underline{z}(x)$ explicitly.
    \[
        z_1 = - \lfloor x_1 \rfloor,
    \]
    \[
        z_2 = - \lfloor x_2 + P_1(x_1, z_1) \rfloor,
    \]
    \[
        z_3 = - \lfloor x_3 + P_2(x_1, x_2, z_1, z_2) \rfloor
    \]
    where $\lfloor \cdot \rfloor$ means taking integer parts. 
    In general, for $2 \leq i \leq m$,
    \[
        z_i = - \lfloor x_i + P_{i-1}(x_1, \ldots, x_{i-1}, z_1, \ldots, z_{i-1}) \rfloor.
    \]
    Let $M = \max_{1 \leq i \leq m} |x_i|$. Since $\deg P_i < k$ for every $i$, it follows that
    \[
        |z_1| \leq |x_1| \leq M,
    \]
    \[
        |z_2| \leq |x_2| + |P_1(x_1, z_1)| \leq M + O(M^k) = O(M^k),
    \]
    \[
        |z_3| \leq |x_3| + |P_2(x_1, z_1, x_2, z_2)| \leq M + O((M^k)^k) = O(M^{k^2}).
    \]
    Inductively, we can show that $|z_i| = O(M^{k^{i-1}})$ for $1 \leq i \leq m$. Letting $c_1 = k^{m-1}$, we have our conclusion.
\end{proof}
\begin{corollary}
\label{cor:U_M}
    For $\underline{x} \in \R^m$ and $n \in \N$, let $\underline{z}_{n}(x) = (z_{1, n}, \ldots, z_{m,n}) \in \Z^m$ be such that $\underline{x}^{*n} * \underline{z}_n(x) \in [0, 1)^m$. Let $M > 0$. Then there exists a constant $c_2 = c_2(M) > 0$ such that the following is true: For $\underline{x} \in [-M, M]^m$ and $1 \leq i \leq m$,
    \[
        |z_{i, n}(x)| =  O_M(n^{c_2}).
    \]
\end{corollary}
\begin{proof}
    By \cref{lem:nov-20-1}, there exists polynomial $R \in \R[x]$ such that
    \[
        \underline{x}^{*n} = (Q_{1, n}(\underline{x}), \ldots, Q_{m,n}(\underline{x}))
    \]
    where $Q_{i,n} \in \Q[x_1, \ldots, x_m]$ of degree less than $k^m$ and the sum of absolute values of coefficients of $Q_{i,n}$ is bounded above by $R(n)$. It follows that if $\underline{x} \in [-M, M]^m$, then
    \[
        |Q_{i,n}(\underline{x})| \leq R(n)M^{k^m}.
    \]
    In view of \cref{prop:compute_z_i}, there exists a constant $c_1$ such that for $\underline{x} \in [-M, M]^m$, 
    \[
        |z_{i,n}(\underline{x})| = O((\max_{1 \leq i \leq m}|Q_{i,n}(\underline{x})|)^{c_1}) = O \left( \left(R(n) M^{k^m}\right)^{c_1}\right) = O_{M}(n^{c_2})
    \]
    where $c_2 = c_1 \deg R$.
\end{proof}

\begin{proposition}
\label{prop:partition_MM}
    Let $M > 0$. There exists constant $c_3 = c_3(M)$ such that the following holds: For every $0 < \epsilon < 1$, if $1 \leq n_1 < n_2 < \ldots < n_{\ell} \leq n$ are natural numbers, then the union of solutions of $\ell(\ell-1)/2$ equations
    \begin{equation}
    \label{eq:ell_things}
        \lVert \varphi(\underline{x}^{* n_i}) - \varphi(\underline{x}^{* n_j}) \rVert_{\T^m} = \epsilon \;\; \; (1 \leq i < j \leq \ell)
    \end{equation}
    partition $[-M, M]^m$ into $O_{M}(n^{c_3})$ regions.
\end{proposition}
\begin{proof}
    Suppose $\varphi = (\varphi_1, \ldots, \varphi_m)$ where $\varphi_i: \R^m \to [0,1)$ for $1 \leq i \leq m$.
    By definition of the metric $\lVert \cdot \rVert_{\T^m}$, for $a, b \leq n$, solutions of the equation
    \[
        \lVert \varphi(\underline{x}^{*a}) - \varphi(\underline{x}^{*b}) \rVert_{\T^m} = \epsilon
    \]
    are solutions of at least one of $m$ equations
    \begin{equation}
    \label{eq:underline}
        \lVert \varphi_i(\underline{x}^{*a}) - \varphi_i(\underline{x}^{*b}) \rVert_{\T} = \epsilon
    \end{equation}
    for $1 \leq i \leq m$.

    Since $\varphi_i(\underline{x}) \in [0, 1)$ all $\underline{x} \in \R^m$, solutions of \eqref{eq:underline} are solutions of at least one of two equations:
    \begin{equation}
    \label{eq:pm_epsilon}
        \varphi_i(\underline{x}^{*a}) - \varphi_i(\underline{x}^{*b}) = \pm \epsilon. 
    \end{equation}
    Because $\varphi(\underline{x}) = \underline{x}*\underline{z}(x)$, \eqref{eq:pm_epsilon} is equivalent to
    \begin{multline}
    \label{eq:expanding_varphi}
        Q_{i,a}(\underline{x}) + z_{i,a}(\underline{x}) + P_{i-1}(Q_{1,a}(\underline{x}), \ldots, Q_{i-1, a}(\underline{x}), z_{1, a}(\underline{x}), \ldots, z_{i-1,a}(\underline{x}))  \\
        - Q_{i,b}(\underline{x}) - z_{i,b}(\underline{x}) - P_{i-1}(Q_{1,b}(\underline{x}), \ldots, Q_{i-1, b}(\underline{x}), z_{1, b}(\underline{x}), \ldots, z_{i-1,b}(\underline{x})) = \pm \epsilon.
    \end{multline}
    
    By \cref{cor:U_M}, there exists a constant $c_2 = c_2(M)$ such that for $\underline{x} \in [-M, M]^m$ and $i \in \{1, 2, \ldots, m\}$, we have
    \[
        |z_{i, a}(\underline{x})| = O_M(a^{c_2})
    \]
    and
    \[
        |z_{i, b}(\underline{x})| = O_M(b^{c_2}).
    \]
    By replacing $z_{i,a}$ and $z_{i,b}$ with possible integer values that they can take, the solution set of \eqref{eq:expanding_varphi} is a subset of the union of solutions of $2 (2 O_M(a^{c_2}) + 1)^i (2 O_M(b^{c_2}) + 1)^i = O_M((ab)^{mc_2}) = O_M(n^{2 m c_2})$ equations:
    \begin{multline}
    \label{eq:expanding_more}
        Q_{i,a}(\underline{x}) + d_{i,a} + P_{i-1}(Q_{1,a}(\underline{x}), \ldots, Q_{i-1, a}(\underline{x}), d_{1,a}, \ldots, d_{i-1,a})  \\
        - Q_{i,b}(\underline{x}) - d_{i,b} - P_{i-1}(Q_{1,b}(\underline{x}), \ldots, Q_{i-1, b}(\underline{x}), d_{1, b}, \ldots, d_{i-1,b}) = \pm \epsilon
    \end{multline}
    where $d_{j,a}, d_{j,b}$ are integers such that $d_{j,a} \in [-O_M(a^{c_2}), O_M(a^{c_2})]$ and $d_{j,b} \in [-O_M(b^{c_2}), O_M(b^{c_2})]$ for $j = 1, 2, \ldots, i$.
    
    It follows that the union of solutions of $\ell(\ell-1)/2$ equations in \eqref{eq:ell_things}
    is a subset of the union of solutions of $\ell (\ell - 1)/2 \times O_M(n^{2 m c_2})$ polynomial equations
    \begin{multline}
    \label{eq:expanding_more_next}
        Q_{i,a}(\underline{x}) + d_{i,a} + P_{i-1}(Q_{1,a}(\underline{x}), \ldots, Q_{i-1, a}(\underline{x}), d_{1,a}, \ldots, d_{i-1,a})  \\
        - Q_{i,b}(\underline{x}) - d_{i,b} - P_{i-1}(Q_{1,b}(\underline{x}), \ldots, Q_{i-1, b}(\underline{x}), d_{1, b}, \ldots, d_{i-1,b}) = \pm \epsilon
    \end{multline}
    where 
    \[
        d_{1,{n_{j_1}}}, \ldots, c_{i,{n_{j_1}}} \in [-O_M(n_{j_1}^{c_2}), O_M(n_{j_1}^{c_2})] \cap \Z,
    \]
    and
    \[
        d_{1,{n_{j_2}}}, \ldots, d_{i,{n_{j_2}}} \in [-O_M(n_{j_2}^{c_2}), O_M(n_{j_2}^{c_2})] \cap \Z
    \]
    for $1 \leq j_1 < j_2 \leq \ell$. 
    
    Since $\ell \leq n$, we have $\ell (\ell - 1)/2 \times O_M(n^{2 m c_2}) = O_M(n^{2 m c_2 + 2})$. Moreover, because every $P_{i-1}$ has degree less than $k$ and every $Q_{i,j}$ has degree less than $k^m$, the left hand side of \eqref{eq:expanding_more_next} is a polynomial of degree less than $k^{km}$.
    
    Observe that the number of regions that the union of solutions of $\ell(\ell - 1)/2$ equations in \eqref{eq:ell_things} partition $[-M, M]^m$ is not greater than the following number: The number of regions that the union of solutions of $\ell(\ell - 1)/2$ equations that appear in \eqref{eq:expanding_more_next} and of $2 m$ equations $x_i = \pm M$ for $ i \in \{1, 2, \ldots, m\}$ partition $\R^m$. In total there are at most $O_M(n^{2m c_2}) + 2m = O_M(n^{2m c_2})$ polynomial equations in which the degree of each polynomial is less than $k^{km}$. According to \cref{lem:poly_partition_space}, the number of regions that the union of solutions of these equations partition $\R^m$ is less than 
    \[
        \left( 2 k^{km} O_M(n^{2m c_2}) \right)^m = O_M(n^{2 c_2 m^2}). 
    \]
     Letting $c_3 = 2 c_2 m^2$, we have our conclusion. 
\end{proof}


To state the next proposition, we need a definition. 
\begin{definition}
\label{def:1}
    Let $A, B \subset \N$ and $g \in G$. For $\epsilon > 0$, we say $A$ and $B$ are \emph{$\epsilon$-separable by $g$} if
    \[
        d(g^A 1_X, g^B 1_X) := \inf_{a \in A, b \in B} d(g^a 1_X, g^b 1_X) \geq \epsilon.
    \]
\end{definition}

\begin{proposition}
\label{prop:not_epsilon_separable}
    Let $E \subset \N$ be a sublacunary set. Then for every $H \subset G$ compact and $\epsilon > 0$, there exist two finite disjoint subsets $A, B \subset E$ such that $A$ and $B$ are not $\epsilon$-separable by any $g \in H$.
\end{proposition}
\begin{proof}
    By contradiction, assume there exist $H \subset G$ compact and $\epsilon > 0$ such that every two finite disjoint subsets of $E$ are $\epsilon$-separable by some $g \in H$. Since $H$ is compact, it is contained inside $\psi([-M, M]^m)$ for some $M$ sufficiently large. Without loss of generality, assume $H = \psi([-M, M]^m)$. 
    
    
    Suppose $E = \{r_n\}_{n \in \N}$. For $N \in \N$, let $R = R_N = \{r_1, r_2, \ldots, r_N\}$. We say a subset $A \subset R$ an \emph{$R$-nice set} if $A$ and $R \setminus A$ are $\epsilon$-separable by some $g \in H$. By our assumption, every subset of $R$ is an $R$-nice set. In other words, there are $2^{|R|} = 2^N$ $R$-nice sets. We will show this is impossible when $N$ is sufficiently large. 
    
    For an $R$-nice set $A$, there exists $g \in H$ such that $d(g^A 1_X, g^{R \setminus A} 1_X) \geq \epsilon$. Associate that $g$ to $A$ by saying \emph{$g$ creates $A$}. We will count number of $R$-nice sets created by $g$ as $g$ ranges over $H$.
    
    For $g \in H$, the set $g^R 1_X$  is a finite subset of $X$. Connect two elements of $g^R 1_X$ if their distance is less than $\epsilon$. Since the metric on $X$ comes from the metric $\lVert \cdot \rVert_{\T^m}$ on $[0,1)^m$, the number of connected components is not greater than $1/\epsilon^m$. 
    
    If $A$ is an $R$-nice set created by $g$, then $A$ must solely consist of some of those connected components. More precisely, if $a, b \in R$ such that $a \in A$ and $g^a 1_X, g^b 1_X$ are in the same connected component, then $b$ must also belong to $A$. (Otherwise, if $b \in R\setminus A$, then on the connected path from $g^a 1_X$ to $g^b 1_X$, there would be two adjacent points whose distance is less than $\epsilon$ but one belongs to $g^A 1_X$ while the other belongs to $g^{R\setminus A} 1_X$. This would contradict the assumption that $d(g^A 1_X, g^{R \setminus A} 1_X) \geq \epsilon$.) Therefore, the number of $R$-nice sets created by a fixed $g$ is at most the number of ways to pick connected components of $g^R 1_X$, which is not greater than $2^{(1/\epsilon)^m}$. 
    
    For $g, h \in H$, $g$ and $h$ create the same collection of $R$-nice sets if for every $a, b \in R$, 
    \[
        d(g^a 1_X, g^b 1_X) \geq \epsilon \text{ if and only if } d(h^a 1_X, h^b 1_X) \geq \epsilon.
    \]
    Or equivalently,
    \[
        \lVert \tilde{\psi} (g^a 1_X) - \tilde{\psi} (g^b 1_X) \rVert_{\T^m} \geq \epsilon \text{ if and only if } \lVert \tilde{\psi}(h^a 1_X) - \tilde{\psi}(h^b 1_X) \rVert_{\T^m} \geq \epsilon.
    \]
    
    Consider $N(N-1)/2$ equations $\lVert \tilde{\psi}(g^a 1_X) - \tilde{\psi}(g^b 1_X) \rVert_{\T^m} = \epsilon$ for $a < b \in R$ (In these equations, the unknown is $g$). The solutions of these equations partitions $H$ into disjoint regions. Note that the map $g \mapsto \lVert \tilde{\psi}(g^a 1_X) - \tilde{\psi}(g^b 1_X) \rVert_{\T^m}$ is continuous, and hence, every $g$ in the same region creates the same collection of $R$-nice sets. It remains to count the number of regions. 
    
    Let $\underline{x} = \psi^{-1}(g)$. Then the equation $\lVert \psi(g^a) - \psi(g^b) \rVert_{\T^m} = \epsilon$ becomes
    \begin{equation}
    \label{eq:abx}
        \lVert \varphi(\underline{x}^{*a}) - \varphi(\underline{x}^{*b}) \rVert_{\T^m} = \epsilon.
    \end{equation}
    Invoking \cref{prop:partition_MM}, there is a constant $c_3 = c_3(M)$ such that the solutions of $N(N-1)/2$ equations of the form \eqref{eq:abx} where $a, b \in R$ with $a < b$ partition $[-M, M]^m$ into $O_M(r_N^{c_3})$ regions.

    Thus there are at most $2^{(1/\epsilon)^m} \times O_M(r_N^{c_3}) = O_{\epsilon, M}(r_N^{c_3})$ distinct $R$-nice sets created by some $g \in H$. Since $E$ is sublacunary (i.e. $(\log r_N)/N \to 0$ as $N \to \infty$), 
    \[
        \lim_{N \to \infty} \frac{\log r_N^{c_3}}{\log 2^N} = 0.
    \]
    It follows that $O_{\epsilon, M}(r_N^{c_3}) < 2^N$ for sufficiently large $N$. This contradicts our assumption that there are always $2^N$ $R$-nice sets for every $R \subset E$ of the form $R = \{r_1, r_2, \ldots, r_N\}$.
\end{proof}
    
\subsection{Proof of the main theorem}    
We are ready to prove \cref{thm:main}.
\begin{proof}[Proof of \cref{thm:main}]
   Let $E = \{r_n\}_{n \in \N}$ be a sublacunary set. By \cref{thm:criterion_bohr_interpolation}, it suffices to show that there exist two disjoint subset $A , B \subset E$ that are not separabe by any $k$-step nilrotation. 
   
   Similar to the proof of \cref{lem:nilrecurrence_partition}, we use the fact that up to isomorphism, the set of pair $(G, \Gamma)$ where $G$ is a $k$-step nilpotent Lie group and $\Gamma$ is a discrete, cocompact subgroup is countable.
   Let $\mathcal{C}$ be a set of representatives of those non-isomorphic pairs $(G, \Gamma)$. For each $(G, \Gamma) \in \mathcal{C}$, we can partition $G$ into countably many compact subsets, say $G = \bigcup_{j=1}^{\infty} H_{G,j}$. It follows that the set $\mathcal{D} = \{(G, \Gamma, H_{G,j}, \epsilon): (G, \Gamma) \in \mathcal{C}, j \in \N, \epsilon \in \{1/n:n \in \N\} \}$ is countable. Enumerate $\mathcal{D}$ as $\{(G_n, \Gamma_n, H_n, \epsilon_n): n \in \N\}$.
   
    
    By \cref{prop:not_epsilon_separable}, there exist $A_1, B_1 \subset E$ finite and disjoint that are not $\epsilon_1$-separable by any $g \in H_1$. Let $E_2 = E \setminus (A_1 \cup B_1)$. Since $A_1, B_1$ are finite, $E_2$ is still sublacunary. Hence there exist $A_2, B_2 \subset E_2$ finite and disjoint that are not $\epsilon_2$-separable by any $g \in H_2$. In general, by induction, we obtain two sequences of finite and pairwise disjoint sets $\{A_n\}_{n \in \N}$ and $\{B_n\}_{n \in \N}$ such that for each $n \in \N$, $A_n$ and $B_n$ are not $(\epsilon_n)$-separable by any $g \in H_n$. Let $A = \bigcup_{n=1}^{\infty} A_n$ and $B = \bigcup_{n=1}^{\infty} B_n$. Then $A$ and $B$ are disjoint subsets of $E$ by construction. 
    
    Note that if $A$ and $B$ are separable by a $k$-step nilrotation, there must be an $\epsilon > 0$ and a $k$-step nilsystem $(X, g)$  such that $A$ and $B$ are $\epsilon$-separable by $g$.
    Since $\bigcup_{n=1}^{\infty} H_n = \bigcup_{n=1}^{\infty} G_n$ and $\lim_{n \to \infty} 1/n = 0$, it follows from our construction that $A$ and $B$ are not separable by any nilrotation $g \in \bigcup_{n=1}^{\infty} G_n$. Because nilrotations on $\bigcup_{n=1}^{\infty} G_n$ represent all possible $k$-step nilrotations, our proof finishes.
\end{proof}

\section{New interpolation sets for Bohr almost periodic sequences}

In this section, we prove \cref{prop:union_of_lacunary_set_with_a_shift} about new examples of $I_0$-sets and the corresponding corollaries. 

\begin{proof}[Proof of \cref{prop:union_of_lacunary_set_with_a_shift}]
    First of all, if $\{t_n\}_{n \in \N}$ is a set of Bohr recurrence, then $\{r_n\}_{n \in \N}$ and $\{r_n + t_n\}_{n \in \N}$ are not separable by any Bohr rotation. Hence, in view of \cref{thm:criterion_bohr_interpolation}, the set $E = \{r_n\}_{n \in \N} \cup \{r_n + t_n\}_{n \in \N}$ is not an $I_0$-set.
    
    Suppose $\{t_n\}_{n \in \N}$ is not a set of Bohr recurrence. Then there exists $\alpha$ in some finite dimensional torus $\T^d$ and $\epsilon > 0$ such that $\lVert t_n \alpha \rVert_{\T^d} > \epsilon$ for all $n \in \N$.  Cover $\T^d$ by finitely many balls of diameter $\epsilon/2$, say $\T^d = \bigcup_{i = 1}^{\ell} B_i$. Then for  each $n \in \N$, $r_n \alpha$ and $(r_n + t_n) \alpha$ must belong to two disjoint balls.
    
    Partition $E$ into $2 \ell$ sets $F_i, F_i'$ for $1 \leq i \leq \ell$ as follows. First, let $F_1 = \{m \in E: m \alpha \in B_1\}$. Then, for each $n \in \N$, if $r_n \in F_1$, let $r_n + t_n \in F_1'$, and if $r_n + t_n \in F_1$ let $r_n \in F_1'$.
    Inductively, for $2 \leq i \leq \ell - 1$, let $E_{i+1} = E \setminus \bigcup_{k=1}^{i} (F_k \cup F_k')$ and $F_{i+1} = \{m \in E_{i+1}: m \alpha \in B_{i+1}\}$. For each $n \in \N$, if $r_n \in F_{i+1}$, let $r_n + t_n \in F_{i+1}'$, and if $r_n + t_n \in F_{i+1}$, let $r_n \in F_{i+1}'$.
    
    By construction, $F_1, \ldots, F_{\ell}, F_1', \ldots, F_{\ell}'$ are pairwise disjoint. Moreover, since $\T^d = \bigcup_{i=1}^{\ell} B_i$, we have $E \alpha  \subset \bigcup_{i=1}^{\ell} B_i$. It follows that $E = \bigcup_{i = 1}^{\ell} (F_i \cup F_i')$. Furthermore, since for every $n \in \N$, $r_n \alpha$ and $(r_n + t_n) \alpha$ belong to two disjoint balls, there are no $n \in \N$ and $A \in \{F_1, \ldots, F_{\ell}, F_1', \ldots, F_{\ell}'\}$ such that both $r_n$ and $r_n + t_n$ are elements of $A$. Hence, from our hypothesis that $\{r_n\}_{n \in \N}$ is lacunary and $\lim_{n \to \infty} t_n /r_n = 0$, each $A \in \{F_1, \ldots, F_{\ell}, F_1', \ldots, F_{\ell}'\}$ is lacunary. In particular, $A$ is an $I_0$-set. To show that $E$ is $I_0$, it remains to show that for $A, B \in \{F_1, \ldots, F_{\ell}, F_1', \ldots, F_{\ell}'\}$ distinct, $A$ and $B$ are separable by a Bohr rotation. 
    
    For $i, j \in \{1, \ldots, \ell\}$ distinct, there does not exist any $n \in \N$ such that both $r_n$ and $r_n + t_n$ belong to $F_i \cup F_j$. (Since if $r_n \in F_i$, then from the construction, $r_n + t_n$ must belong to $F_i'$, not $F_j$. Likewise, if $r_n + t_n \in F_i$, $r_n$ must be in $F_i'$, not $F_j$.) Hence, the set $F_i \cup F_j$ is lacunary. Because every lacunary set is an $I_0$-set, $F_i, F_j$ are separable by a Bohr rotation according to \cref{thm:criterion_bohr_interpolation}. By the same reason, $F_i,  F_j'$ and $F_i', F_j'$ are separable by Bohr rotations. 
    
    For $i \in \{1, \ldots, \ell\}$, $F_i \alpha \in B_i$ while $F_i' \alpha$ is in the complement of the ball having the same center as $B_i$ but with twice diameter. Hence, $F_i \alpha$ and $F_i' \alpha$ have disjoint closures in $\T^d$. In particular, $F_i$ and $F_i'$ are separable by $\alpha$.   
    
    We conclude that every two sets in $\{F_1, \ldots F_{\ell}, F_1', \ldots, F_{\ell}'\}$ are separable by a Bohr rotation. Therefore, $E$ is an $I_0$-set. 
\end{proof}

\begin{corollary}
\label{cor:2-step_Bohr}
    Let $\{r_n\}_{n \in \N} \subset \N$ be lacunary and let $\{t_n\}_{n \in \N} \subset \N$ be such that $\lim_{n \to \infty} t_n/r_n = 0$. Then $\{r_n\}_{n \in \N} \cup \{r_n + t_n\}_{n \in \N}$ is a $2$-step interpolation set. 
\end{corollary}
\begin{proof}
   Note that if $(\psi(n))_{n \in \N}$ is an $1$-step nilsequence, $(\psi(n^2))_{n \in \N}$ is a $2$-step nilsequence.  
   It follows that a set $\{s_n\}_{n \in \N}$ is a $2$-step interpolation set if $\{s_n^2\}_{n \in \N}$ is an $1$-step interpolation set (See the proof of \cite[Proposition 1.7]{Le_2019_1}).

   Let $\{r_n\}_{n \in \N}$ and $\{t_n\}_{n \in \N}$ satisfy the hypothesis of the corollary. Consider the set $F = \{r_n^2\}_{n \in \N} \cup \{(r_n + t_n)^2\}_{n \in \N} = \{r_n^2\}_{n \in \N} \cup \{r_n^2 + 2r_n t_n + t_n^2\}_{n \in \N}$. Since $\lim_{n \to \infty} t_n/r_n = 0$ and $\{t_n\}_{n \in \N}$ is increasing, we have 
   \[
    \liminf_{n \to \infty} \frac{2 r_{n+1} t_{n+1} + t_{n+1}^2}{2 r_n t_n + t_n^2} \geq \liminf_{n \to \infty} \frac{2 r_{n+1} + t_{n+1}}{2 r_n + t_n} \geq \liminf_{n \to \infty} \frac{r_{n+1}}{r_n}.
   \]
   Because $\{r_n\}_{n \in \N}$ is lacunary, $\{2r_n t_n + t_n^2\}_{n \in \N}$ is also lacunary. Hence it is not a set of Bohr recurrence \cite{Furstenberg81a}. Furthermore, we also have
   \[
    \lim_{n \to \infty} \frac{2 r_n t_n + t_n^2}{r_n^2} = 0.
    \]
    
    In view of \cref{prop:union_of_lacunary_set_with_a_shift}, the set $F$ is an $I_0$-set. From our discussion at the beginning, the set $\{r_n\}_{n \in \N} \cup \{r_n + t_n\}_{n \in \N}$ is a $2$-step interpolation set.
\end{proof}

\cref{cor:2-step-1-step} now follows from \cref{prop:union_of_lacunary_set_with_a_shift} and \cref{cor:2-step_Bohr}. Next we obtain another corollary of \cref{prop:union_of_lacunary_set_with_a_shift}.
\begin{corollary}[{\cite[Corollary 2.7.9]{Graham_Hare_2013}}]
\label{cor:lac_union_shift}
    Let $E \subset \N$ be lacunary and let $F \subset \N \cup \{0\}$ be finite. Then $E + F := \{m + n: m \in E, n \in F\}$ is an $I_0$-set.
\end{corollary}
\begin{proof}
   By Hartman-Ryll-Nardzewski criterion (\cref{thm:criterion_bohr_interpolation}), it suffices to show $E+i$ and $E+j$ are separable by a Bohr rotation for any $i \neq j \in F$. But this follows from \cref{prop:union_of_lacunary_set_with_a_shift} with $t_n = j - i$ for $n \in \N$.
\end{proof}

\begin{remark}\label{rem:lacunary_plus_finite_set}
In contrast with \cref{cor:lac_union_shift}, it is not true in general that the sum of an $I_0$-set and a finite set is an $I_0$-set. For example, let $E = \{2^n\}_{n \in \N} \cup \{2^n + 2n - 1\} _{n \in \N}$. Then $E$ is an $I_0$-set. However, $E + \{0, 1\} = \{2^n\}_{n \in \N} \cup \{2^n + 2n - 1\}_{n \in \N} \cup \{2^n + 1\}_{n \in \N} \cup \{2^n + 2n\}_{n \in \N}$ is not an $I_0$-set because $\{2^n\}_{n \in \N}$ and $\{2^n+2n\}_{n \in \N}$ are not separable by any Bohr rotation.
\end{remark}

\section{Open questions}
As shown in \cite{Le_2019_1} and also \cref{cor:2-step-1-step}, there are $2$-step interpolation sets that are not $I_0$-sets. It is natural to ask to what extent this feature holds for higher step nilsequences. The following question has been asked in \cite{Le_2019_1} and is currently still open.

\begin{question}
    For $k \in \N$, does there exist a $(k+1)$-interpolation set that is not a $k$-interpolations set?
\end{question}

A topological system $(X, T)$ with a metric $d$ on $X$ is called \emph{distal} if for every $x, y \in X$ distinct, $\inf_{n \in \N} d(T^n x, T^n y) > 0$. A \emph{distal sequence} is a sequence of the form $(F(T^nx))_{n \in \N}$ for $n \in \N$ where $F$ is a continuous function on a distal system $(X, T)$ and $x \in X$. It is shown in \cite{Auslander_Green_Hahn63} that nilsystems are distal, therefore every nilsequence is a distal sequence. It is then of interest to study the notion of interpolation sets to distal sequences. 
\begin{definition}
    A set $E \subset \N$ is called an \emph{interpolation set for distal sequences} if for every bounded function $f: E \to \C$, there exists a distal sequence $\psi: \N \to \C$ such that $\psi(n) = f(n)$ for $n \in E$. 
\end{definition}
Based on \cref{thm:main} and \cref{thm:union_with_finite_set}, there are two questions we can ask.
\begin{question}
    Let $E \subset \N$ be an interpolation set for distal sequences and $F \subset \N$ finite. Is it true that $E \cup F$ is an interpolation set for distal sequences?
\end{question}
\begin{question}
    Is it true that no sublacunary set is an interpolation set for distal sequences?
\end{question}

\bibliography{refs}
\bibliographystyle{plain}
\end{document}